\documentclass [12pt] {article}
\usepackage{amsmath,amssymb,amsbsy,amsthm}
\usepackage[utf8]{inputenc}
\usepackage{graphicx}
\usepackage{multirow}
\usepackage{tikz}
\usepackage{tkz-graph}
\usetikzlibrary{shapes}
\tikzstyle{vertex}=[circle, draw, inner sep=2pt, minimum size=6pt]

\newtheorem{observation}{Observation}
\newtheorem{thm} {\indent Theorem}

\def\zet{\mathop\mathbb{Z}\nolimits}

\title{Orientable $\mathbb{Z}_{n}$-distance magic regular graphs}
\author{ Paweł Dyrlaga, Karolina Szopa \\ AGH University of Science and Technology Kraków, Poland}

\begin{document}
\maketitle

\begin{abstract}
Hefetz, M\"{u}tze, and Schwartz  conjectured that every connected undirected graph admits an antimagic orientation \cite{ref_HefMutSch}. In this paper we support  the analogous question for distance magic labeling. Let $\Gamma$ be an Abelian group of order $n$. A \textit{directed $\Gamma$-distance magic labeling} of an oriented graph $\vec{G} = (V,A)$ of order $n$ is a bijection $\vec{l}:V \rightarrow \Gamma$ with the property that there is a \textit{magic constant} $\mu \in \Gamma$
such that for every $x \in V(G)$
$
w(x) = \sum_{y \in N^{+}(x)}\vec{l}(y) - \sum_{y \in N^{-}(x)} \vec{l}(y) = \mu.
$ 
In this paper we provide an infinite family of odd regular graphs possessing an orientable $\zet_n$-distance magic labeling. Our results refer to lexicographic product of graphs. We also present a family of odd regular graphs that are not orientable $\mathbb{Z}_{n}$-distance magic. 
\end{abstract}

\section{Introduction}

Consider a simple graph $G$ and a simple oriented graph  $\vec{G}$. We denote by $V(G)$ the vertex set and by $E(G)$ the edge set of $G$. For $\vec{G}$ we denote by $V(\vec{G})$ the vertex set and by $A(\vec{G})$ the arc set of $\vec{G}$. We denote the order of $G$ by $|V(G)|$. An arc $\overrightarrow{xy}$ is considered to be directed from $x$ to $y$, moreover $y$ is called the $head$ and $x$ is called the $tail$ of the arc.
For a vertex $x$, the set of head endpoints adjacent to $x$ is denoted by $N^{-}(x)$, and the set of tail endpoints adjacent to $x$ is denoted by $N^{+}(x)$. For graph theory notations and terminology not described in this paper, the readers are referred to \cite{Har}.

In this paper we investigate distance magic labelings, which belong to a large family of magic-type labelings.  Probably the best known problem in  the area of magic and antimagic labelings is the {\em antimagic conjecture} by Hartsfield and Ringel~\cite{HarRin}, which claims that the edges of every graph except $K_2$ can be labeled by integers $1,2,\dots,|E|$ so that the weight of each vertex is different. The conjecture is still open. Twenty years later Hefetz, M\"{u}tze, and Schwartz introduced the variation of antimagic labelings, i.e.,
antimagic labelings on directed graphs. Moreover, they conjectured that every connected undirected graph admits an antimagic orientation \cite{ref_HefMutSch}. The papers \cite{CicFreFro,FreKer, FreKer2} stated the analogous question for distance magic labeling, namely when a graph $G$ of order $n$ has a $\zet_n$-distance magic orientation. 

Formally speaking, a \textit{directed $\Gamma$-distance magic labeling} of an oriented graph $\vec{G} = (V,A)$ of order $n$ is a bijection $\vec{l}:V \rightarrow \Gamma$ with the property that there is a \textit{magic constant} $\mu \in \Gamma$
such that for every $x \in V(G)$

$$
w(x) = \sum_{y \in N^{+}(x)}\vec{l}(y) - \sum_{y \in N^{-}(x)} \vec{l}(y) = \mu,
$$ 
where the sum is taken in the group $\Gamma$, and instead of writing $a+(-b)$ we use a $a-b$ notation. 

If for a graph $G$ there exists an orientation $\vec{G}$ such that there is a directed $\Gamma$-distance magic labeling $\vec{l}$ for $\vec{G}$, we say that $G$ is \textit{orientable $\Gamma$-distance magic} and the directed $\Gamma$-distance magic labeling $\vec{l}$ we call an \textit{orientable $\Gamma$-distance magic labeling}.

Cichacz, Freyberg and Froncek proved the following theorems.

\begin{thm}[\cite{CicFreFro}]\label{gr:odd1}Let $G$ have order $n\equiv2\pmod4$ and all vertices of odd degree.
There does not exist an orientable $\zet_n$-distance magic labeling
of $G$.
\end{thm}

\begin{thm}[\cite{CicFreFro}]\label{complete} The complete graph  $K_n$ is orientable $\mathbb{Z}{}_{n}$-distance magic if and only if $n$ is odd.
\end{thm}

\begin{thm}[\cite{CicFreFro}]\label{cff}
Let $G = K_{n_{1},n_{2}, \ldots, n_{k}}$ be a complete $k$-partite graph such that $1 \leq n_{1} \leq n_{2} \leq \ldots \leq n_{k}$ and $n = n_{1}+n_{2}+ \ldots + n_{k}$ is odd. The graph $G$ is orientable $\mathbb{Z}_{n}$-distance magic graph if $n_{2} \geq 2$.
\end{thm}

In this paper we consider two out of four standard graph products \linebreak (see \cite{HIK}). The \textit{Cartesian product} $G {\displaystyle \square } H$ of graphs $G$ and $H$ is a graph such that the vertex set of $G {\displaystyle \square } H$ is the Cartesian product $V(G) \times V(H)$ and any two vertices $(g,h)$ and $(g',h')$ are adjacent in $G {\displaystyle \square } H$ if and only if either $g=g'$ and $h$ is adjacent with $h'$ in $H$ or $h=h'$ and $g$ is adjacent with $g'$ in $G$.

The \textit{lexicographic product} $G \circ H$ is a graph with the vertex set \linebreak $V(G) \times V(H)$. Two vertices $(g,h)$ and $(g',h')$ are adjacent in $G \circ H$ if and only if either $g$ is adjacent with $g'$ in $G$ or $g = g'$ and $h$ is adjacent to $h'$ in $H$. $G\circ H$ is also \textit{called the composition of graphs} $G$ and $H$ and denoted by $G[H]$ (see \cite{Har}).

We will also discuss join graphs. We say that $G$ is a \textit{join graph }if $G$ is the complete union of two graphs
$G_1 = (V_1, E_1)$ and $G_2 = (V_2, E_2)$. In other words, $V = V_1 \cup V_2$ and $E = E_1 \cup E_2 \cup \{uv : u \in V_1, v \in V_2\}$. If $G$ is
the join graph of $G_1$ and $G_2$, we shall write $G = G_1 + G_2$.

Freyberg and Keranen proved recently.
\begin{thm}[\cite{FreKer2}]\label{FreKer2}
If $H$ is an orientable $\mathbb{Z}{}_{m}$-distance magic graph of order $m$ and $n\not \equiv 2 \pmod 4$, then the
lexicographic product $G = H \circ \overline{K_{n}}$ is orientable  $\mathbb{Z}{}_{mn}$-distance magic.
\end{thm}

So far there was known only one example of odd regular graph  orientable $\mathbb{Z}_{n}$-distance magic (\cite{CicFreFro}). Note that by Theorems~\ref{complete} and \ref{FreKer2} for $m$ odd and $n\not \equiv 2 \pmod 4$ the product  
$K_m \circ \overline{K_{n}}\cong K_{\underbrace{n,n,\ldots,n}_{m}} $ 
is orientable $\zet_{mn}$-distance magic.
In Section~\ref{complete}  we give necessary and sufficient conditions for $K_m\circ \overline{K_{n}}$ being orientable $\zet_{mn}$-distance magic. As a consequence, we provide an infinite family of odd regular graphs possessing directed $\zet_n$-distance magic labeling. Moreover in Section~\ref{join} we will show one more example of a family of graphs that is orientable $\mathbb{Z}{}_{n}$-distance magic if and only if $n \not \equiv 2 \pmod 4$. In the last section we present some family of odd regular graphs that are not orientable $\mathbb{Z}{}_{n}$-distance magic. Before we proceed, we will need the following theorem.
\begin{thm}[\cite{Zeng}]\label{set}
Let $n = r_1 + r_2 + \ldots + r_q$ be a partition of the positive even  integer $t$, where $r_i\geq 2$  for $i = 1,2,\ldots,q$. Let $A = \{-t,-(t-1),\ldots,-1,1,\ldots,(t-1),t\}$. Then the set $A$ can be partitioned into pairwise disjoint subsets $A_1,A_2,\ldots,A_q$ such that for every $1 \leq i \leq q$, $|A_i| = r_i$ with
$\sum_{a\in A_i} a =0$.
\end{thm}

\section{Lexicographic products}

Consider a graph $G = K_m \circ \overline{K_{n}}$. Let us denote independent sets of vertices by $V^{1}, V^{2}, \ldots, V^{m}$ where $V^{i} = \{v_{1}^{i}, v_{2}^{i}, \ldots, v_{n}^{i} \}$ for  $i \in \{1,2,\ldots, m\}$. We give necessary and sufficient conditions
for $K_m \circ \overline{K_n}$ being orientable $\zet_{mn}$-distance magic. 

\begin{thm}
A graph $G = K_{m} \circ \overline{K_{n}}$ is orientable $\mathbb{Z}_{mn}$-distance magic if and only if $n \not \equiv 1 \pmod{2}$ or $m \not \equiv 2 \pmod{4}$ for $n \geq 2$. When $n = 1$, $G$ is orientable $\mathbb{Z}_{mn}$-distance magic if and only if $m$ is odd.
\end{thm}

\begin{proof}
By Theorem \ref{complete} we can assume that $n>1$. If $mn$ is odd, then we are done by Theorem \ref{cff}. Moreover from Theorem \ref{gr:odd1} we can conclude that for $n \equiv 1 \pmod{2}$ and $m \equiv 2 \pmod{4}$ there does not exist an orientable $\mathbb{Z}_{mn}$-distance magic labeling for the graph $G$. We consider two cases:

\textit{ Case 1: } $mn \equiv 0\pmod4$.\\
Let $A = \{-\frac{mn}{2}+1,  -\frac{mn}{2}+2, \ldots, -1,1, \ldots, \frac{mn}{2}-2, \frac{mn}{2}-1 \} = \mathbb{Z}_{mn} \setminus \{0, \frac{mn}{2} \}$.

If $n=2$, then let  $A^1=\{0,\frac{mn}{2} \}$ and there exists a zero-sum partition $A^2,A^3,\ldots,A^m$  of the set $A$ such that  $|A^i| =2$ for every $2 \leq i \leq m$ by Theorem \ref{set}. Let $q$ be the index of subset containing $\frac{mn}{4}$.

If $n>2$, then by Theorem \ref{set} there exists a partition of $A$ into $A_{1}', A_{2}', \ldots, A_{m}'$ such that $|A_{1}'|=|A_{2}'|=n-1$, $|A_{i}'|=n$ for $i = 3,4, \ldots, m$ and $\sum_{a \in A_{i}'} a = 0 $ for $i=1,2, \ldots, m$. Without loss of generality $\frac{mn}{4} \in A_{q}$, where $q \neq 1$.  Let $A^{1} = A_{1}' \cup \{ \frac{mn}{2}\}$, $A^{2}=A_{2}' \cup \{0\}$, $A^{i} = A_{i}'$, where $i=3,4, \ldots, m$.

Note that in both situations $\sum_{a \in A^{i}}a \equiv 0 \pmod{mn}$ for $i=2,3, \ldots, m$. Label vertices from each set $V^{i}$ by the elements of $A^{i}$ with the restriction that $\vec{l}(v_{1}^{1}) = \frac{mn}{2}$ and $\vec{l}(v_{1}^{q})= \frac{mn}{4}$.

Let $o(uv)$ be the orientation for the edge $uv$. For edges $v^{1}_{i}v^{q}_{j}$ from $E(G)$ we have
\begin{equation}
o(v^{1}_{i}v^{q}_{j})= \begin{cases} \overrightarrow{v^{1}_{i}v^{q}_{j}}, \quad i \in \{1,2, \ldots, n \}, j=1 \\
 							 \overrightarrow{v^{q}_{j}v^{1}_{i}}, \quad  i \in \{1,2, \ldots, n \}, j \in \{2,3, \ldots, n \}. \end{cases} 
\end{equation}
This way we obtained the edge orientation between $V^{1}$ and $V^{q}$. For the remaining pairs of partition vertex sets $V^{k}$ and $V^{l}$ we demand all edges to be oriented the same from the $k$-th set to the $l$-th set or conversely. One can check that $w(v) \equiv \frac{mn}{2} \pmod{mn}$ for any $v \in V(G)$. Thus $G$ is orientable $\mathbb{Z}_{mn}$-distance magic.

\textit{ Case 2:} $m \equiv 1\pmod4$ and $n \equiv 2\pmod4$.\\
For the set $V^{k}$ we introduce the following orientation 
$$o(v_{i}^{k}v_{j}^{l}) = \overrightarrow{v_{j}^{l}v_{i}^{k}}$$
for all $l \in \{k - \frac{m-1}{2}, k - \frac{m-1}{2} +1, \ldots, k-1\}$ (where operations are taken modulo $m$). For the remaining edges we have
$$o(v_{i}^{k}v_{j}^{l}) = \overrightarrow{v_{i}^{k}v_{j}^{l}}.$$

We say that $V^l$ \textit{preceds} $V^k$ and $V^k$ \textit{succeeds} $V^l$ if arcs between vertices of $V^l$ and $V^k$ have tails in $V^l$ and heads in $V^k$. Define the labeling $\vec{l}$ such that 
$\vec{l}(v_{i}^{k}) = (i-1) + (k-1) n$, where $i \in \{1,2, \ldots, n \}$ and $k \in \{1,2, \ldots, m\}$. It is easy to see that for each $k \in \{1, 2, ..., m\}, i \in \{1,2,...,n\}$ we have

$$  w(v_i^k) = \sum_{\substack{v \in V^{l}: \\ V^{l} \  preceds \  V^{k}}} \vec{l}(v) - \sum_{\substack{v \in V^{l}: \\ V^{l}\  succeeds \  V^{k}}}\vec{l}(v) = \frac{m-1}{2} n d \pmod{mn},$$
where $d$ is a constant difference between labels $a_{i}^{j} \in A^{j}$ and $a_{i}^{j'} \in A^{j'}$ where \linebreak $j' = j + \frac{m+1}{2} \pmod{m}$. Therefore $d = n \frac{m+1}{2}$. We obtain the magic constant $\mu \equiv \frac{m-1}{2} \frac{m+1}{2} n^2 \pmod{mn}$.
\end{proof}

Observe that if $G$ is an odd regular graph, then  the lexicographic product $G\circ \overline{K}_{2n+1}$ is also an odd regular graph. From the above Theorem 5 we obtain the following observation showing that there exist infinitely many odd regular graphs that are orientable $\Gamma$-distance magic for a cyclic group $\Gamma$.

\begin{observation}
The lexicographic product $K_{4m}\circ \overline{K}_{2n+1}$ has a $\zet_{4m(2n+1)}$-distance magic labeling for any $m,n\geq1$.
\end{observation} 

Note that the method presented above works also for other families of graphs, for instance for $K_{1,4m+3}\circ  \overline{K}_{n}$. We just assign the label $\frac{4(m+1)n}{2}$ to some vertex in the center of the $K_{1,4m+3}\circ  \overline{K}_{n}$ and place $\frac{4(m+1)n}{4}$ in the other set. The orientation in $K_{1,4m+3}\circ  \overline{K}_{n}$ is similar to the general case.

\section{Join graphs}\label{join}
\begin{thm}
The join graph $G=P_{n-1}+K_1$ is orientable $\mathbb{Z}_{n}$-distance magic if and only if $n \not \equiv 2 \pmod 4$.
\end{thm}
\begin{proof}
Let $v_1, v_2, \ldots, v_{n-1}$ be consecutive vertices belonging to $P_{n-1}$, and let $v_n$ be the remaining vertex of $G$. 

First we prove that when $n \equiv 2 \pmod{4}$ $G$ is not orientable $\mathbb{Z}_n$-distance magic. 
We proceed by contradiction and assume that $G$ is orientable $\mathbb{Z}_n$-distance magic. Let $\vec{l}$ be some orientable $\mathbb{Z}_n$-distance  magic labeling of $G$ and let $\mu$ be the magic constant. Because the group $\mathbb{Z}_n$ has an even order, it makes sense to speak about parity. Observe that $\mu \pmod{2} \equiv w(v_n) \pmod{2} \equiv (\vec{l}(v_1) + \ldots + \vec{l}(v_{n-1})) \pmod{2} \equiv \left(\sum_{z \in \mathbb{Z}_{n}}z - \vec{l}(v_n)\right) \pmod{2} \equiv (1 + \vec{l}(v_n)) \pmod{2}$. Therefore the parity of $\vec{l}(v_n)$ is always opposite to the parity of $\mu$. 

Since $w(v_1) \pmod{2} \equiv (\vec{l}(v_n) + \vec{l}(v_2)) \pmod{2} \equiv (\vec{l}(v_n)+1) \pmod{2}$ we get $\vec{l}(v_2) \equiv 1 \pmod{2}$. Moreover $w(v_3) \equiv (\vec{l}(v_2)+\vec{l}(v_4)+\vec{l}(v_n)) \pmod{2}$ which implies $\vec{l}(v_4) \equiv 0 \pmod{2}$,
and in general $\vec{l}(v_j) \equiv 1 \pmod{2}$ \linebreak for $j \equiv 2\pmod{4}$, and $\vec{l}(v_j) \equiv 0 \pmod{2}$ for $j \equiv 0 \pmod{4}$. On the other hand
$w(v_{n-1}) \equiv (\vec{l}(v_n-2)+\vec{l}(v_n)) \pmod{2}$, so $\vec{l}(v_{n-2}) \equiv 1 \pmod{2}$, but $n-2 \equiv 0 \pmod{4}$.  
Contradiction. In other cases we show that $G$ is orientable $\mathbb{Z}_n$-distance magic.

We set the following labeling. When $n$ is odd we assign $\vec{l}(v_i) = i$  for $i = 1, \ldots, n-1$  and $\vec{l}(v_n) = 0$. 
Now we set orientation.
$$
	o(v_{i}v_{i+1}) = \overrightarrow{v_{i+1}v_{i}}
$$
and 
$$
	o(v_{i}v_{n}) = \overrightarrow{v_{i}v_{n}},
$$
for $i = 1, \ldots, n-2$.
And finally $o(v_{n-1}v_{n}) = \overrightarrow{v_{n}v_{n-1}}$.

Observe that when $i = 2, \ldots, n-2$ we get $w(v_i) = i+1 -(i-1) - 0 = 2$.
Next, $w(v_1) = 2 - 0 = 2$, $w(v_{n-1}) = -(n-2) + 0 = 2$, and finally $w(v_n) = 1 + 2 + \ldots + (n-2) - (n-1) = 1 -n + 1 = 2$.
This way we obtain magic constant $\mu = 2$.

When $n \equiv 0 \pmod{4}$ we change the orientation of two arcs: $v_nv_{n-1}$ to $v_{n-1}v_n$
and $v_sv_n$ to $v_nv_s$, where $s = \frac{n}{4} - 1$. We set the same labeling as in the previous case. Therefore $w(v_i)$, where $i \in \{1, \ldots, n-1\} \setminus \{s\}$, can be calculated similarly. Next, $w(v_s) = s+1 - (s-1) + 0 = 2$, and 
$
	w(v_n) = 1 + 2 + \ldots + (s-1) -s + (s+1) + \ldots + (n-1) = 1 + \ldots + (n-1) - 2s = \frac{n}{2} - 2(\frac{n}{4} - 1) = 2.
$
And we get the magic labeling with $\mu = 2$.
\end{proof}

\section{Prism graphs}

 In this section we present some odd regular graphs that are not orientable distance magic. The \textit{prism} is a Cartesian product $P_{2} {\displaystyle \square } C_{n}$.

\begin{thm}
Let us consider a \textit{prism graph} $G$ of order $2n$. There does not exist an orientable $\zet_{2n}$-distance magic labeling.
\end{thm}

\begin{proof}
If $|G|=|P_{2} {\displaystyle \square } C_{n}| \equiv 2 \pmod{4}$ the thesis is fulfilled on the basis of Theorem 1. We are going to consider situation when $|G| \equiv 0 \pmod{4}$. Suppose that there exists an oriented $\mathbb{Z}_{2n}$-distance magic labeling for a graph $G$. Since $G$ is bipartite graph we can assume that it has partition sets $U= \{u_{1}, u_{2}, \ldots, u_{2k} \}$ and $W = \{w_{1}, w_{2}, \ldots ,w_{2k} \}$.

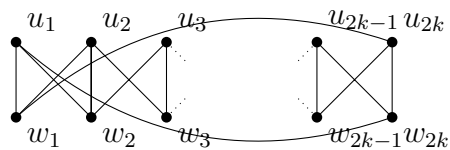
\begin{center}
\begin{tikzpicture}
 \fill (0,0) circle(2pt) node[below right] {$w_{1}$};
 \fill (0,1) circle(2pt) node[above right] {$u_{1}$};
 \fill (1,1) circle(2pt) node[above right] {$u_{2}$};
 \fill (1,0) circle(2pt) node[below right] {$w_{2}$};
 \fill (2,1) circle(2pt) node[above right] {$u_{3}$};
 \fill (2,0) circle(2pt) node[below right] {$w_{3}$};
 \fill (4,0) circle(2pt) node[below right] {$w_{2k-1}$};
 \fill (4,1) circle(2pt) node[above right] {$u_{2k-1}$};
 \fill (5,0) circle(2pt) node[below right] {$w_{2k}$};
 \fill (5,1) circle(2pt) node[above right] {$u_{2k}$};
 \draw (0,1)--(1,0)--(1,1)--(2,0)--(2,1)--(1,0);
 \draw (0,0)--(1,1)--(1,0)--(1,1);
 \draw (0,0)--(0,1); \draw[dotted] (2,1)--(2.25,0.75); \draw[dotted] (2,0)--(2.25,0.25);
 \draw[dotted] (4,0)--(3.75,0.25); \draw[dotted] (4,1)--(3.75,0.75); \draw (4,0)--(4,1);
 \draw (4,0)--(5,1)--(5,0)--(4,1);
 \draw (0,0) to[bend left] (5,1);
 \draw (5,0) to[bend left] (0,1);

\end{tikzpicture}
\end{center}

As in Section 3 since $\mathbb{Z}_{2n}$ has an even order, it makes sense to speak about even and odd elements. Let us focus on the parity of the magic constant. There is no need to consider direction of edges because addition and subtraction modulo $2$ give the same results. If we know the parity of three consecutive labels $\vec{l}(u_{i-1})$, $\vec{l}(u_{i}), \vec{l}(u_{i+1})$, then we can say what is the parity of the element $\vec{l}(u_{i+2})$. The parity of the magic constant generated by three consecutive labels needs to be preserved, which means:
\begin{center}
$(\vec{l}(u_{i}) + \vec{l}(u_{i+1}) + \vec{l}(u_{i+2})) \equiv (\vec{l}(u_{i+1}) + \vec{l}(u_{i+2})+\vec{l}(u_{i+3})) \pmod{2}$
\end{center}
Therefore, $\vec{l}(u_{i}) \equiv \vec{l}(u_{i+3}) \pmod{2}$ for any $i \in \{1,2,\ldots,2k\}$.

Hence knowing the parity of three initial labels one can establish parity of the remaining labels. We examine two possibilities, according to the cardinality of $U$. 
\\Suppose first, that $|U| \not \equiv 0 \pmod{3}$. One can check that $\vec{l}(x) \equiv \vec{l}(y) \pmod{2}$ for any $x,y$ from the same partition set. 
Thus, \linebreak $\vec{l}(u) \equiv 1+\vec{l}(w) \pmod{2}$ for $u \in U$ and $w \in W$. Because the graph is odd regular the parity of the magic constant depends on the partition set, a contradiction.
\\
We assume now that $|U| \equiv 0 \pmod{3}$, then we have to examine every three initial labels generating the rest of the sequence. For labels of the same parity we have contradiction that is compatible with the description above. If one of the three initial $\vec{l}(u_{1}),\vec{l}(u_{2}),\vec{l}(u_{3})$ is an even number and the rest are odd numbers then in the whole $U$ component we have $\frac{1}{3}$ of all even numbers and $\frac{2}{3}$ of odd numbers. They generate even magic constant. On the other hand, in $W$ component we have $\frac{1}{3}$ of all odd numbers and $\frac{2}{3}$ of all even numbers (the rest of remaining labels). This does not allow a labeling that generates even magic constant because vertices from $W$ also should meet the rule of the same parity on every third element of the sequence. Therefore, there also should be $\frac{1}{3}$ of all even numbers and $\frac{2}{3}$ of odd numbers or all labels were even. That situation is not possible. 
\\
The case looks similar in the scheme with one odd number and two even numbers in initial three $\vec{l}(u_{1}),\vec{l}(u_{2}),\vec{l}(u_{3})$. Above consideration exhausts other possible cases and therefore proves the rightness of  the formula.  

\end{proof}

\section*{Acknowledgements}
We would like to thank Sylwia Cichacz for her support, encouragement, assistance in proofreading this reasearch and delivering valuable tips and resources. We could not have imagined having a better advisor and mentor. We are also very grateful to Dominika Dato\'n, Kinga Patera, Natalia Pondel, Maciej Gabry\'s and Przemysław Zietek from \textit{"Snark" Research Student Association} for their help and involvement in initial phase of our analysis.

\end{document}